\documentclass[english]{amsart}
\usepackage{a4wide}

\newtheorem{theorem}{Theorem}[section]
\newtheorem{lemma}[theorem]{Lemma}
\newtheorem*{theorem*}{Theorem}
\newtheorem{remark}[theorem]{Remark}

\newtheorem{proposition}[theorem]{Proposition}

\newtheorem{definition}[theorem]{Definition}

\numberwithin{equation}{section}

\theoremstyle{remark}

\usepackage{graphicx}

\usepackage[fixlanguage]{babelbib}
\selectbiblanguage{english}

\usepackage[shortlabels]{enumitem}

\usepackage{amsfonts,amssymb,amsbsy,amsmath,mathrsfs,amsthm}

\usepackage{commath}

\allowdisplaybreaks

\usepackage{hyperref}
\usepackage{url}

\hypersetup{pdfpagemode=UseNone, pdfstartview=FitH,
  colorlinks=true,urlcolor=red,linkcolor=blue,citecolor=black,
  pdftitle={Default Title, Modify},pdfauthor={Your Name},
  pdfkeywords={Modify keywords}}

\RequirePackage{babel}

\providecommand{\norm}[1]{ \lVert#1  \rVert}

\newcommand{\dx}{\, d x}
\newcommand{\dy}{\, d y}
\newcommand{\dz}{\, d z}
\newcommand{\dt}{\, d t}

\def\Xint#1{\mathchoice
   {\XXint\displaystyle\textstyle{#1}}%
   {\XXint\textstyle\scriptstyle{#1}}%
   {\XXint\scriptstyle\scriptscriptstyle{#1}}%
   {\XXint\scriptscriptstyle\scriptscriptstyle{#1}}%
   \!\int}
\def\XXint#1#2#3{{\setbox0=\hbox{$#1{#2#3}{\int}$}
     \vcenter{\hbox{$#2#3$}}\kern-.5\wd0}}

\def\dashint{\Xint-}

\DeclareMathOperator*{\esssup}{ess\,sup}
\DeclareMathOperator*{\essinf}{ess\,inf}

\usepackage{cite}
\makeatletter
\newcommand{\citecomment}[2][]{\citen{#2}#1\citevar}
\newcommand{\citeone}[1]{\citecomment{#1}}
\newcommand{\citetwo}[2][]{\citecomment[,~#1]{#2}}
\newcommand{\citevar}{\@ifnextchar\bgroup{;~\citeone}{\@ifnextchar[{;~\citetwo}{]}}}
\newcommand{\citefirst}{\@ifnextchar\bgroup{\citeone}{\@ifnextchar[{\citetwo}{]}}}

\makeatother

\def\avgint{\Xint-}
\newcommand{\rn}{{{\mathbb R}^n}}
\newcommand{\R}{{\mathbb R}}

\begin{document}

\title{Two-weight norm inequalities for parabolic fractional maximal functions}

\author{David Cruz-Uribe, OFS}
\address{Department of Mathematics \\
University of Alabama \\
 Tuscaloosa, AL 35487, USA}
\email{dcruzuribe@ua.edu}

\author{Kim Myyryl\"ainen}
\address{Department of Mathematical Analysis, Charles University, Sokolovsk\'a 83, 186 75 Prague 8, Czech Republic}
\email{kim.myyrylainen@matfyz.cuni.cz}

\thanks{The first author is partially supported by a Simons Foundation
  Travel Support for Mathematicians Grant and by NSF Grant DMS-2349550. 
  The second author is supported by Charles University PRIMUS/24/SCI/020 and Research Centre program No. UNCE/24/SCI/005. The authors would like to thank the referees for their careful reading and helpful comments.}

\date{\today}



\begin{abstract}

We prove two-weight norm inequalities for  parabolic fractional maximal functions using parabolic Muckenhoupt weights.
In particular, we prove a two-weight, weak-type estimate 
and Fefferman--Stein type inequalities for the centered parabolic maximal function.
We also prove that a parabolic Sawyer-type condition implies the strong-type estimate for the parabolic fractional maximal function.
Finally, we prove the strong-type estimate 
for the centered parabolic maximal function
assuming a stronger parabolic Muckenhoupt bump condition.

\end{abstract}

\maketitle

\section{Introduction}

In this paper we prove two-weight, weak and strong-type norm inequalities for parabolic maximal operators.  To put our results into context, we first review some prior results, beginning with those for the Hardy--Littlewood maximal operator and the fractional maximal operator.   The first results of this kind are due to Muckenhoupt~\cite{muckenhoupt1972}, who proved that given a pair of weights $(w,v)$, a necessary and sufficient condition for the Hardy--Littlewood maximal operator,
\[ Mf(x) = \sup_{Q} \avgint_Q |f(y)|\,dy \cdot \chi_Q(x),  \]
to satisfy the weak $(q,q)$ inequality, $1\leq q < \infty$,
\[ w(\{ x\in \rn : Mf(x) > \lambda \}) \leq \frac{C}{\lambda^q} \int_\rn |f(x)|^q v\dx, \]
is that $(w,v)$ satisfy the two-weight $A_q$ condition: if $q>1$,
\[ [(w,v)]_{A_q} = \sup_Q \avgint_Q w \dx \bigg( \avgint_Q v^{1-q'}\dx \bigg)^{q-1} < \infty, \]
and if $q=1$,
\[ [(w,v)]_{A_1} = \sup_Q \esssup_{x\in Q} v(x)^{-1} \avgint_Q w\dy < \infty. \]
It is well-known (though not originally proved in the literature it follows from arguments in~\cite{MR340523}--see~\cite{MR3642364}) that a similar result holds for the fractional maximal operator,
\[ M_\alpha f(x) = \sup_{Q} |Q|^{\frac{\alpha}{n}} \avgint_Q |f(y)|\,dy \cdot \chi_Q(x),  \]
where $0<\alpha<n$: for $1\leq q< \frac{n}{\alpha}$ and $\frac{1}{q}-\frac{1}{r}=\frac{\alpha}{n}$,
\[ w^r(\{ x\in \rn : M_\alpha f(x) > \lambda \}) \leq \frac{C}{\lambda^r} \bigg(\int_\rn |f|^q v^q\dx \bigg)^{\frac{r}{q}}, \]
if and only if $(w,v)$ satisfy the two-weight $A_{q,r}$ condition: if $q>1$,
\[ [(w,v)]_{A_{q,r}} = \sup_Q \avgint_Q w^r\,dx \bigg(\avgint_Q v^{-q'}\dx\bigg)^{\frac{r}{q'}} <\infty,  \]
 and if $q=1$,
\[ [(w,v)]_{A_{1,r}} = \sup_Q \esssup_{x\in Q} v(x)^{-r} \avgint_Q w^r\,dy < \infty.  \]
(Here and below, the supremum is taken over all cubes $Q$ with sides parallel to the coordinate axes.)
Previously, a special case of the $(1,1)$ inequality was proved by Fefferman and Stein~\cite{feffermanstein1971}.  They proved that when $q=1$, the weak $(1,1)$ inequality holds for pairs of weights  of the form $(w,Mw)$, where $w\in L^1_{\mathrm{loc}}(\rn)$; consequently, by Marckinkiewicz interpolation they proved that when $q>1$, the strong $(q,q)$ inequality holds for the same pairs.  

A characterization of two-weight, strong-type inequalities for fractional maximal operators was first proved by Sawyer~\cite{sawyer1982}, who introduced the so-called testing conditions.  He showed that if $0\leq \alpha <n$, $1<q\leq r < \infty$, then 
\begin{equation} \label{eqn:strong-qr}
 \bigg(\int_\rn (M_\alpha f)^r w\,dx\bigg)^{\frac{1}{r}} 
\leq 
C \bigg(\int_\rn |f|^q v\,dx\bigg)^{\frac{1}{q}} 
\end{equation}
if and only if 
\[ [(w,v)]_{S_{q,r,\alpha}} = \sup_Q \bigg(\int_Q v^{1-q'}\,dx\bigg)^{-\frac{1}{q}}
\bigg(\int_Q M_\alpha(v^{1-q'}\chi_Q)^r w\,dx\bigg)^{\frac{1}{r}} < \infty.  
\]
This condition is referred to as a testing condition since it consists of the norm inequality for the family of test functions $v^{1-q'}\chi_Q$.

Since the strong-type $(q,r)$ implies the corresponding weak $(q,r)$ inequality, we have that the $A_{q,r}$ is necessary for the strong $(q,r)$ inequality to hold.  However, examples show that it is not sufficient:  see~\cite{muckenhoupt1972,MR3642364}.  P\'erez~\cite{perez1994,perez1995} introduced a generalization of the $A_{q,r}$ condition, generally referred to as a ``bump condition."  He showed that a sufficient condition for \eqref{eqn:strong-qr} to hold is that the weights satisfy, for some $s>1$,
\[ [(w,v)]_{A_{q,r,s}} = \sup_Q |Q|^{\frac{\alpha}{n}+\frac{1}{r}-\frac{1}{q}} \bigg(\avgint_Q w\,dx\bigg)
\bigg(\avgint_Q v^{s(1-q')}\,dx\bigg)^{\frac{r}{q's}} < \infty.
\]
We note that he proved substantially stronger results, with the second norm replaced by a smaller norm in the scale of Orlicz spaces (e.g., with a so-called "log-bump").  We refer the reader to the original papers or to~\cite{MR2797562} for more information on bump conditions.

The analog of the Fefferman--Stein inequality for the fractional maximal operator was proved by Sawyer~\cite{MR670111}:  he proved that if  $0\leq \alpha <n$ and $1<q<\frac{n}{\alpha}$, then
\[ \int_\rn (M_\alpha f)^q w\,dx \leq C\int_\rn |f|^q M_{\alpha q}w\,dx.   \]

\ 

\medskip

We now consider the generalization of these results to one-sided operators on the real line.  We define the 
 one-sided fractional maximal operators on $\R$ as follows:  for $0\leq \alpha<1$, let
\[ M_\alpha^+ f(x) = \sup_{h>0} \frac{1}{h^{1-\alpha}}\int_x^{x+h} |f(y)|\,dy.
\]
For weak-type inequalities, Mart\'{\i}n-Reyes and de la Torre~\cite{MR1425871} (see also~\cite{martinpicktorre1993}) proved the following.  If $1\leq q< \frac{1}{\alpha}$ and $\frac{1}{q}-\frac{1}{r}=\alpha$,
\[ w^r(\{ x\in \R : M_\alpha f(x) > \lambda \}) \leq \frac{C}{\lambda^r} \bigg(\int_\R |f|^q v^q\dx \bigg)^{\frac{r}{q}}, \]
if and only if $(w,v)$ satisfy the two-weight $A_{q,r}^+$ condition: if $q>1$,
\[ [(w,v)]_{A_{q,r}^+} = \sup_{a<b<c} (c-a)^{\alpha-1}\bigg(\int_a^b w^r\,dx\bigg)^{\frac{1}{r}}
\bigg(\int_b^c v^{-q'}\,dx\bigg)^{\frac{1}{q'}} < \infty, 
\]
and if $q=1$,
\[  [(w,v)]_{A_{1,r}^+} = \sup_{a<b} v(b)^{-1}\bigg( (a-b)^{-1} \int_a^b w^{r}\,dx\bigg)^{\frac{1}{r}} < \infty. 
\]
For strong-type inequalities, Mart\'{\i}n-Reyes and de la Torre~\cite{martinreyestorre1993} (see also~\cite{sawyer1986,martinsalvadortorre1990, MR930071}) showed that for $0\leq \alpha <1$ and $1<q\leq r<\infty$, the strong-type inequality
 \[ \bigg(\int_\R (M_\alpha^+ f)^r w\,dx\bigg)^{\frac{1}{r}} 
\leq 
C \bigg(\int_\R |f|^q v\,dx\bigg)^{\frac{1}{q}} 
\]
holds if and only if 
\begin{equation*} 
[(w,v)]_{S_{q,r,\alpha}^+} = \sup_I \bigg(\int_{I^+} v^{1-q'}\,dx\bigg)^{-\frac{1}{q}}
\bigg(\int_I M_\alpha^+(v^{1-q'}\chi_{I^+})^r w\,dx\bigg)^{\frac{1}{r}} < \infty.  
\end{equation*}
Here the supremum is taken over all intervals $I=[a,c]$ with $I^+=[b,c]$ for all $a<b<c$.

Two-weight, Fefferman--Stein type inequalities for one-sided maximal operators were proved by de Rosa~\cite{derosa2006}, who showed that if $0\leq \alpha <1$ and $1<q<\frac{1}{\alpha}$, then
\[ \int_\R (M^+_\alpha f)^q w\,dx \leq C\int_\R |f|^q M^-_{\alpha q}w\,dx.   \]

\ 

\medskip

Since the development of the weighted theory for one-sided operators in the real line, there have been a number of attempts to extend this theory to $\rn$, $n>1$.  We refer the reader to~\cite{berra2022,ombrosi2005,ForzaniMartinreyesOmbrosi2011,MR1330244,berkovits2011} for further information.   More recently, one-sided operators have been generalized to $\rn$ in a different direction, using the parabolic geometry that arises in the study of parabolic differential equations.  The one-weight theory was introduced by Kinnunen and Saari~\cite{kinnunenSaariParabolicWeighted,kinnunenSaariMuckenhoupt} for a one-sided maximal operator; see also~\cite{KinnunenMyyry2023,KinnunenMyyryYangZhu2022,KinnunenMyyryYang2022,saari2018,saari2016}.  The one-weight theory for fractional maximal operators was introduced by Ma, He, and Yan~\cite{MaHeYan2023}.

Our main results in this paper are to extend this approach to begin to develop a two-weight theory for parabolic maximal operators.  We will state our results here, but for brevity will defer some definitions to Section~\ref{sec:twoweightMuckenhoupt}.  
In particular, we work in the parabolic geometry of parabolic rectangles $R(x,t,L,p)$ with $1\leq p<\infty$. The parameter $p$ while implicit is important in establishing underlying geometry and will be fixed from now on.

Our first result is a weak $(q,r)$ inequality for the centered parabolic maximal operator.

\begin{theorem}
\label{thm:twoweightweaktype}
Let $0\leq\gamma<1$, $1\leq q<\infty$,
and $0\leq\alpha<\frac{1}{q}$.  Define $r$ by $\frac{1}{q}-\frac{1}{r}=\alpha$, 
and let $(w,v) \in A^+_{q,r}(\gamma)$:  that is, if $q>1$,
\[
[(w,v)]_{A^+_{q,r}(\gamma)} = \sup_{R \subset \mathbb R^{n+1}} \biggl( \dashint_{R^-(\gamma)} w^r \,dx\dt \biggr)^\frac{1}{r} 
\biggl( \dashint_{R^+(\gamma)} v^{-q'} \,dx\dt  \biggr)^\frac{1}{q'} < \infty,
\]
and if $q=1$,
\[
[(w,v)]_{A^+_{1,r}(\gamma)} = \sup_{R\subset \R^{n+1}} \big(\essinf_{(x,t)\in R^+(\gamma)} v(x,t)\big)^{-1}
\biggl( \dashint_{R^-(\gamma)} w^r \,dx\dt \biggr)^\frac{1}{r} <\infty.
\]
Then there is a constant $C=C(n,p,\gamma,q,r,\alpha)$ such that
\[
w^r( \{ M^{\gamma+}_{\alpha, c} f > \lambda \}) 
\leq [(w,v)]_{A_{q,r}^+(\gamma)}^r \frac{C}{\lambda^r}
\biggl( \int_{\mathbb{R}^{n+1}} \lvert f \rvert^q \, v^q \,dx\dt \biggr)^\frac{r}{q}
\]
for 
every $f\in L^1_{\mathrm{loc}}(\mathbb{R}^{n+1})$.

\end{theorem}

\begin{remark}
    For our proof we are required to work with the centered fractional maximal operator.  It is an open question to extend our results to the general fractional maximal operator $M_{\alpha}^{\gamma+}$.  Note that in the parabolic geometry these operators are no longer equivalent.
\end{remark}

We can also prove several strong-type results.  The first is a characterization in terms of Sawyer-type testing conditions.

\begin{theorem}
\label{thm:twoweightstrongsawyer}
Let $1<q\leq r<\infty$, $0\leq\alpha<1$
and $w,\,v$ to be weights.  Then 
\[
[(w,v)]_{S_{q,r,\alpha}^+} = 
\sup_R \biggl(  \int_{R^+} v^{1-q'} \,dx\dt\biggr)^{-\frac{1}{q}}
\biggl( \int_{R} (M^{+}_{\alpha}(v^{1-q'} \chi_{R^+}))^r \, w \,dx\dt \biggr)^\frac{1}{r} < \infty
\]
for every parabolic rectangle $R\subset\mathbb{R}^{n+1}$ if and and only if 
there is a constant $C=C(n,p,q,r)$ such that
\[
\biggl( \int_{\mathbb{R}^{n+1}} (M_{\alpha,c}^{+}f)^r \, w \,dx\dt\biggr)^\frac{1}{r} 
\leq
C [(w,v)]_{S_{q,r,\alpha}^+} \biggl( \int_{\mathbb{R}^{n+1}} \lvert f \rvert^q \, v \,dx\dt \biggr)^\frac{1}{q}
\]
for 
every $f\in L^1_{\mathrm{loc}}(\mathbb{R}^{n+1})$.
\end{theorem}

\begin{remark}
    In Theorem~\ref{thm:twoweightstrongsawyer} we are only able to prove the Sawyer-type characterization when the time lag $\gamma=0$.  It is an open problem to show that this characterization also holds for $\gamma>0$.
\end{remark}

\begin{remark}
    Mart\'{\i}n-Reyes and de la Torre~\cite{martinreyestorre1993} actually showed that their one-sided testing condition is equivalent to assuming that it holds with $I^+$ replaced everywhere by $I$.  It is an open question as whether we can replace $R^+$ by $R$ everywhere in the $S_{q,r,\alpha}^+$ condition.
\end{remark}

\begin{remark}
    In the one-weight case (i.e., if $v=w$), when $q>1$, we have that the $A_{q,r}^+$ condition is sufficient for the strong $(q,r)$ inequality to hold.  Thus this condition and the $S_{q,r,\alpha}^+$ are equivalent.  It would be interesting to have a direct proof of this fact that did not pass through the strong-type inequality.  In the classical case this was proved by Hunt, {\em et al.}~\cite{MR730066}.
\end{remark}

Our second strong-type result is a bump condition, similar to those proved by P\'erez discussed above.

\begin{theorem} 
\label{thm:twoweightstrongbump}
Let $0\leq\gamma<1$, $1<q,\,s<\infty$ and $w,v$ to be nonnegative measurable functions.
Assume that 
\[
[(w,v)]_{A^+_{q,q,s}(\gamma)} = 
\sup_{R \subset \mathbb R^{n+1}} \biggl( \dashint_{R^-(\gamma)} w \,dx\dt \biggr) 
\biggl( \dashint_{R^+(\gamma)} v^{\frac{s}{1-q}} \,dx\dt \biggr)^\frac{q-1}{s} < \infty .
\]
Then there is a constant $C=C(n,p,\gamma,q,s)$ such that
\[
\int_{\mathbb{R}^{n+1}} (M_c^{\gamma+}f)^q \, w \,dx\dt
\leq C [(w,v)]_{A_{q,q,s}^+(\gamma)} \int_{\mathbb{R}^{n+1}} \lvert f \rvert^q \, v \,dx\dt
\]
for every $f\in L^1_{\mathrm{loc}}(\mathbb{R}^{n+1})$.
\end{theorem}

\begin{remark}
    To the best of our knowledge, this result is the first one-sided bump condition to appear in the literature, even for one-sided operators on the real line.  The problem of proving off-diagonal estimates (i.e., for $q<r$) and proving bump conditions using the Orlicz bumps known for classical operators is open.  Our proof, which uses Theorem~\ref{thm:twoweightweaktype}, only holds for the centered maximal operator; if Theorem~\ref{thm:twoweightweaktype} were true for more general maximal operators, our proof would extend to this setting immediately.
\end{remark}

Finally, we are able to prove a Fefferman--Stein type inequality for the centered fractional maximal operator.

\begin{theorem}
\label{thm:twoweightFS}
Let $0\leq\gamma<1$, $1<q<\infty$
and let $w$ be a weight.
Then there are constants $C_1 = C_1(n,p,\gamma)$ and $C_2=C_2(n,p,\gamma,q)$ such that
\[
w( \{ M_{c}^{\gamma+}f > \lambda \}) \leq \frac{C_1}{\lambda} \int_{\mathbb{R}^{n+1}} \lvert f \rvert \, M^{\gamma-}w\,dx\dt
\]
for 
every $f\in L^1_{\mathrm{loc}}(\mathbb{R}^{n+1})$.
Moreover, 
\[
\int_{\mathbb{R}^{n+1}} (M_c^{\gamma+}f)^q \, w \,dx\dt
\leq C_2  \int_{\mathbb{R}^{n+1}} \lvert f \rvert^q \, M^{\gamma-}w\,dx\dt
\]
for 
every $f\in L^1_{\mathrm{loc}}(\mathbb{R}^{n+1})$.
\end{theorem}

\begin{remark}
One feature of this result is that we need the larger, uncentered parabolic maximal operator on the righthand side of the inequality.  It is an open question if this result is true with the centered maximal operator on the right or with the uncentered maximal operator on the left.
\end{remark}

\begin{remark}
After this article was finished, we found that Kong, Yang, Yuan and Zhu \cite{KYYZ25} also introduced concurrently
the two-weight parabolic Muckenhoupt classes
and studied the two-weight norm inequalities for parabolic fractional maximal functions.
Except for the definition and two basic properties of the two-weight parabolic Muckenhoupt classes, their article and ours have no significant overlap.
\end{remark}

The remainder of this paper is organized as follows.  In Section~\ref{sec:twoweightMuckenhoupt} we gather together the definitions of the parabolic maximal operators that appear in our work, define the two-weight parabolic Muckenhoupt classes, and prove some basic properties analogous to those that hold in the classical case.  In Section~\ref{section:weak} we prove Theorems~\ref{thm:twoweightweaktype} and~\ref{thm:twoweightFS}.  To prove the first result, we use a version of the covering argument in~\cite{ForzaniMartinreyesOmbrosi2011} adapted to the parabolic geometry.
The technique has also been used in~\cite{kinnunenSaariParabolicWeighted,MaHeYan2023,KinnunenMyyryYangZhu2022,KinnunenMyyry2023} for different parabolic maximal functions.  The proof of the second result follows by showing the pairs of weights satisfy a parabolic Muckenhoupt condition and then applying interpolation.  In Section~\ref{section:strong} we prove Theorems~\ref{thm:twoweightstrongbump} and~\ref{thm:twoweightstrongsawyer}. The  proof of Theorem~\ref{thm:twoweightstrongbump} follows by adapting a very general argument due to P\'erez~\cite{perez1993} to our setting.  The proof of Theorem~\ref{thm:twoweightstrongsawyer} is much more technical.  In recent years much progress has been made in harmonic analysis by using dyadic grids in place of arbitrary cubes.  Here, our main contribution is to develop an analogous theory of dyadic rectangles in the parabolic geometry, defining dyadic versions of the parabolic maximal operators, and then relating them to their non-dyadic counterparts. We believe that this machinery will have additional applications in the study of parabolic operators.

\section{two-weight parabolic Muckenhoupt class}
\label{sec:twoweightMuckenhoupt}

Throughout this paper we will use the following notation.
The underlying space that we work on is $\mathbb{R}^{n+1}=\{(x,t):x=(x_1,\dots,x_n)\in\mathbb R^n,t\in\mathbb R\}$.
Unless otherwise stated, constants $C$, $c$, etc. are positive and may change at each appearance.  Any  dependencies on parameters are indicated in parentheses: e.g., $c(n,q)$.
Given a measurable set $A \subset \mathbb{R}^{n+1}$, denote its  Lebesgue measure  by $\lvert A\rvert$.
A (Euclidean) cube $Q$ is a bounded rectangle in $\mathbb R^n$, whose sides parallel to the coordinate axes and have equal length, i.e.,
$Q=Q(x,L)=\{y \in \mathbb R^n: \lvert y_i-x_i\rvert \leq \frac L2,\,i=1,\dots,n\}$
with $x\in\mathbb R^n$ and $L>0$. 
The point $x$ is called the center of the cube and $L$ is sidelength of the cube.
By a weight we mean a nonnegative, measurable function on $\mathbb{R}^{n+1}$. Given a weight $w$ and $A \subset \mathbb{R}^{n+1}$,
we define
\[ w(A) = \int_A w\,dx\dt. \]

We now introduce the parabolic geometry on $\mathbb{R}^{n+1}$.  Instead of working with Euclidean cubes, we will use parabolic rectangles.

\begin{definition}\label{def_parrect}
Let $1\leq p<\infty$, $x\in\mathbb R^n$, $t \in \mathbb{R}$ and $L>0$.
A parabolic rectangle centered at $(x,t)$ with sidelength $L$ is the set
\[
R = R(x,t,L,p) = Q(x,L) \times (t-L^p, t+L^p).
\]
Define its upper and lower parts to be
\[
R^+(\gamma) = Q(x,L) \times (t+\gamma L^p, t+L^p) 
\quad\text{and}\quad
R^-(\gamma) = Q(x,L) \times (t - L^p, t - \gamma L^p) ,
\]
where $0 \leq \gamma < 1$ is called the time lag.
The center point of $R^\pm(\gamma)$ is denoted by $z(R^\pm(\gamma))$.

\end{definition}

Note that the rectangle $R^-(\gamma)$ is the reflection of $R^+(\gamma)$ with respect to the time slice $\mathbb{R}^n \times \{t\}$.
We denote the spatial sidelength of a parabolic rectangle $R$  by $l_x(R)=L$,  and denote the time length by $l_t(R)=2L^p$.
For simplicity, we will write $R^\pm$ instead of $R^{\pm}(0)$.
The top of a rectangle $R = R(x,t,L)$ is $Q(x,L) \times\{t+L^p\}$
and the bottom is $Q(x,L) \times\{t-L^p\}$.
We define the $\lambda$-dilate of $R$  by $\lambda>0$ to be the set $\lambda R = R(x,t,\lambda L^p)$.

Given a measurable set $A\subset\mathbb{R}^{n+1}$ with $0<|A|<\infty$, the integral average of $f \in L^1(A)$ over $A$ is defined to be
\[
f_A = \dashint_A f \dx \dt = \frac{1}{\lvert A\rvert} \int_A f(x,t)\dx\dt .
\]

\medskip

We now define the  parabolic fractional maximal functions and two-weight parabolic Muckenhoupt classes.
We first define the parabolic maximal functions.
Hereafter, when integrating over parabolic rectangles, we will omit the differentials $\dx \dt$.

\begin{definition}
\label{def:parmaximalfct}
Let $0\leq\gamma<1$, $0\leq\alpha<1$, and let $f$ be a locally integrable function. 
The centered, forward in time  parabolic fractional maximal function is defined by
\[
M^{\gamma+}_{\alpha,c} f(x,t) = \sup_{z(R^-(\gamma))=(x,t)}
\lvert R^+(\gamma) \rvert^\alpha
\dashint_{R^+(\gamma)} \lvert f \rvert ,
\]
where the supremum is taken over the parabolic rectangles $R\subset\mathbb{R}^{n+1}$ such that $(x,t)$ is the center point of $R^-(\gamma)$.
Similarly, the centered, backward in time parabolic fractional maximal function is
\[
M^{\gamma-}_{\alpha,c} f(x,t) = \sup_{z(R^+(\gamma))=(x,t)} 
\lvert R^-(\gamma) \rvert^\alpha
\dashint_{R^-(\gamma)} \lvert f \rvert ,
\]
where the supremum is taken over the parabolic rectangles $R\subset\mathbb{R}^{n+1}$ such that $(x,t)$ is the center point of $R^+(\gamma)$.
Moreover, we define the uncentered counterparts by
\[
M^{\gamma+}_\alpha f(x,t) = \sup_{R^-(\gamma)\ni(x,t)} 
\lvert R^+(\gamma) \rvert^\alpha
\dashint_{R^+(\gamma)} \lvert f \rvert
\]
and
\[
M^{\gamma-}_\alpha f(x,t) = \sup_{R^+(\gamma)\ni(x,t)} 
\lvert R^-(\gamma) \rvert^\alpha
\dashint_{R^-(\gamma)} \lvert f \rvert.
\]
If $\gamma=0$, we write $M_{\alpha,c}^\pm f = M_{\alpha,c}^{0\pm} f$ and $M^\pm_\alpha f = M^{0\pm}_\alpha f$.
If $\alpha=0$, we write
$M_{c}^{\gamma\pm} f = M_{0,c}^{\gamma\pm} f$ and $M^{\gamma\pm} f = M^{\gamma\pm}_0 f$.
\end{definition}

We now define our generalization of the two-weight Muckenhoupt classes to the parabolic setting.

\begin{definition}
Let $0\leq\gamma<1$ and $1\leq q\leq r<\infty$.
If $q>1$, a pair of weights $(w,v)$ is in the parabolic Muckenhoupt class $A^+_{q,r}(\gamma)$ if 
\[
[(w,v)]_{A^+_{q,r}(\gamma)} = \sup_{R \subset \mathbb R^{n+1}} \biggl( \dashint_{R^-(\gamma)} w^r  \biggr)^\frac{1}{r} 
\biggl( \dashint_{R^+(\gamma)} v^{-q'}  \biggr)^\frac{1}{q'} < \infty .
\]
If this condition  holds with the time axis reversed (i.e., exchanging $R^+$ and $R^-$), then $(w,v) \in A^-_{q,r}(\gamma)$.
If $q=1$,  $(w,v)$ is in the parabolic Muckenhoupt class $A^+_{1,r}(\gamma)$ if
\[
\biggl( \dashint_{R^-(\gamma)} w^r  \biggr)^\frac{1}{r} \leq C \essinf_{(x,t)\in R^+(\gamma)} v(x,t)
\]
for every parabolic rectangle $R \subset \mathbb R^{n+1}$. Denote the infimum of the constants for which this inequality holds by $[(w,v)]_{A^+_{1,r}(\gamma)}$.
If this condition  holds with the time axis reversed, then $(w,v) \in A^-_{1,r}(\gamma)$.
\end{definition}

We have the following equivalent condition for $A_{1,r}^+(\gamma)$.

\begin{proposition}
\label{thm:maximalA1-cond}
Let $0\leq\gamma<1$.
A pair of weights $(w,v)$  is in $A_{1,r}^+(\gamma)$ if and only if there is a constant $C$ such that 
\begin{equation}
\label{maximalA_1-cond}
(M^{\gamma-}w^r(x,t) )^\frac{1}{r} \leq C v(x,t)
\end{equation}
for a.e. $(x,t) \in \mathbb{R}^{n+1}$.
Moreover, we can take $C = [(w,v)]_{A^+_{1,r}(\gamma)}$.
This result holds for $A_{1,r}^-(\gamma)$ with $M^{\gamma+}$ in place of $M^{\gamma-}$.
\end{proposition}

\begin{proof}
Assume that \eqref{maximalA_1-cond} holds. Then for almost every $(x,t) \in R^+(\gamma)$,
\[
\biggl( \dashint_{R^-(\gamma)} w^r \biggr)^\frac{1}{r} \leq ( M^{\gamma-} w^r(x,t) )^\frac{1}{r} \leq C v(x,t). 
\]
Thus, if we take the essential infimum over every $(x,t) \in R^+(\gamma)$,  we get that $(w,v)\in A_{1,r}^+(\gamma)$.

Now assume that $(w,v)\in A_{1,r}^+(\gamma)$ with the constant $C=[(w,v)]_{A^+_{1,r}(\gamma)}$.
Define
\[
E = \{ (x,t)\in \mathbb{R}^{n+1}: ( M^{\gamma-}w^r(x,t) )^\frac{1}{r} > C v(x,t) \}
\]
and let $(x,t)\in E$.
Then there is a parabolic rectangle $R$ such that $(x,t) \in R^+(\gamma)$ and
\begin{equation}
\label{A1_contra}
\biggl( \dashint_{R^-(\gamma)} w^r \biggr)^\frac{1}{r} > C v(x,t).
\end{equation}
For every $\varepsilon>0$ there exists a rectangle 
$\widetilde{R}$,
whose spatial vertices and the bottom time coordinate have rational coordinates,
such that $(x,t)\in \widetilde{R}^+(\gamma)$, $R^-(\gamma) \subset \widetilde{R}^-(\gamma)$ and $\lvert \widetilde{R}^-(\gamma) \setminus R^-(\gamma) \rvert <\varepsilon$. 
This is possible since the time interval of $R^+(\gamma)$ is open and $\gamma<1$, and thus there is a positive distance between $t$ and the time bottom of $R^+(\gamma)$.
By our choice of  $\widetilde{R}$,
$$\vert \widetilde{R}^-(\gamma) \rvert = \lvert R^-(\gamma) \rvert + \lvert \widetilde{R}^-(\gamma) \setminus R^-(\gamma) \rvert < \lvert R^-(\gamma) \rvert + \varepsilon. $$
If we fix $\varepsilon>0$ sufficiently small, we have that
\[
\biggl( \dashint_{\widetilde{R}^-(\gamma)} w^r \biggr)^\frac{1}{r} 
\geq \biggl( \frac{1}{\lvert R^-(\gamma) \rvert + \varepsilon} \int_{R^-(\gamma)} w^r  \biggr)^\frac{1}{r} > C v(x,t) .
\]
Hence, without loss of generality we may assume that the spatial vertices and the bottom time coordinate of the parabolic rectangles $R$ that satisfy \eqref{A1_contra} are rational.
The $A_{1,r}^+(\gamma)$ condition and \eqref{A1_contra} imply that
\[
C v(x,t) < \biggl( \dashint_{R^-(\gamma)} w^r \biggr)^\frac{1}{r} \leq C \essinf_{(y,s)\in R^+(\gamma)} v(y,s) .
\]
Since $C>0$, it follows that
\[
v(x,t) < \essinf_{(y,s)\in R^+(\gamma)} v(y,s) .
\]

Enumerate the parabolic rectangles in $\mathbb{R}^{n+1}$ with rational spatial vertices and rational bottom time coordinates by  $\{R_i\}_{i\in\mathbb{N}}$. For each $i\in\mathbb{N}$, define
\[
E_i = \Bigl\{(x,t)\in R_i^+(\gamma): v(x,t) < \essinf_{(y,s)\in R_i^+(\gamma)} v(y,s) \Bigr\} .
\]
Then for every $i\in\mathbb{N}$, $\lvert E_i \rvert =0$,  and the argument above shows that $E\subset\bigcup_{i\in\mathbb{N}} E_i$.
Thus,  $\lvert E \rvert =0$ and so \eqref{maximalA_1-cond} follows with $C=[(w,v)]_{A^+_{1,r}(\gamma)}$.
\end{proof}

The next lemma shows that the class of parabolic Muckenhoupt weights is closed if we take maxima and minima.

\begin{lemma}
\label{lem:trunctation}
Let $0\leq\gamma<1$, $1\leq q\leq r<\infty$,  and
$(w_1,v_1), (w_2,v_2)\in A_{q,r}^+(\gamma)$. Then we have $( \max\{w_1,w_2\}, \max\{v_1,v_2\}) \in A_{q,r}^+(\gamma)$ and $( \min\{w_1,w_2\}, \min\{v_1,v_2\}) \in A_{q,r}^+(\gamma)$.

\end{lemma}

\begin{proof}
Let $w = \max\{w_1,w_2\}$ and $v = \max\{v_1,v_2\}$. 
First assume that $1<q<\infty$.  Then
we have
\begin{align*}
&\biggl( \dashint_{R^-(\gamma)} w^r \biggr)^\frac{1}{r} \biggl(
\dashint_{R^+(\gamma)} v^{-q'} \biggr)^\frac{1}{q'} \\
& \qquad =
\biggl(  \frac{1}{\lvert R^-(\gamma) \rvert} \int_{R^-(\gamma) \cap \{w_1>w_2\}} w^r \biggr)^\frac{1}{r} \biggl( \dashint_{R^+(\gamma)} v^{-q'} \biggr)^\frac{1}{q'} \\
&\qquad \qquad +
\biggl(  \frac{1}{\lvert R^-(\gamma) \rvert} \int_{R^-(\gamma) \cap \{w_1 \leq w_2\}} w^r \biggr)^\frac{1}{r} \biggl( \dashint_{R^+(\gamma)} v^{-q'} \biggr)^\frac{1}{q'}
\\
&\qquad\leq
\biggl( \dashint_{R^-(\gamma)} w_1^r \biggr)^\frac{1}{r} \biggl( \dashint_{R^+(\gamma)} v^{-q'} \biggr)^\frac{1}{q'}
+
\biggl( \dashint_{R^-(\gamma)} w_2^r \biggr)^\frac{1}{r} \biggl( \dashint_{R^+(\gamma)} v^{-q'} \biggr)^\frac{1}{q'}
\\
&\qquad\leq
\biggl( \dashint_{R^-(\gamma)} w_1^r \biggr)^\frac{1}{r} \biggl( \dashint_{R^+(\gamma)} v_1^{-q'} \biggr)^\frac{1}{q'}
+
\biggl( \dashint_{R^-(\gamma)} w_2^r \biggr)^\frac{1}{r} \biggl( \dashint_{R^+(\gamma)} v_2^{-q'} \biggr)^\frac{1}{q'}
\\
&\qquad\leq
[(w_1,v_1)]_{A^+_{q,r}(\gamma)} + [(w_2,v_2)]_{A^+_{q,r}(\gamma)} .
\end{align*}
Now suppose that $q=1$; then
\begin{align*}
\biggl( \dashint_{R^-(\gamma)} w^r \biggr)^\frac{1}{r} &\leq \biggl(
\frac{1}{\lvert R^-(\gamma) \rvert} \int_{R^-(\gamma) \cap \{w_1>w_2\}} w^r \biggr)^\frac{1}{r} + \biggl( \frac{1}{\lvert R^-(\gamma) \rvert} \int_{R^-(\gamma) \cap \{w_1\leq w_2\}} w^r \biggr)^\frac{1}{r} \\
&\leq
\biggl( \dashint_{R^-(\gamma)} w_1^r \biggr)^\frac{1}{r} + \biggl( \dashint_{R^-(\gamma)} w_2^r \biggr)^\frac{1}{r} \\
&\leq
[(w_1,v_1)]_{A^+_{1,r}(\gamma)} \essinf_{(x,t)\in R^+(\gamma)} v_1(x,t) + [(w_2,v_2)]_{A^+_{1,r}(\gamma)} \essinf_{(x,t)\in R^+(\gamma)} v_2(x,t) \\
&\leq 
\bigl( [(w_1,v_1)]_{A^+_{1,r}(\gamma)} + [(w_2,v_2)]_{A^+_{1,r}(\gamma)} \bigr) \essinf_{(x,t)\in R^+(\gamma)} v(x,t) .
\end{align*}
If we take the  supremum over all parabolic rectangles $R\subset\mathbb{R}^{n+1}$, we prove the result for maxima.
The proof for  minima is essentially the same.
\end{proof}

\section{weak-type estimates}
\label{section:weak}

In this section we prove Theorems~\ref{thm:twoweightweaktype} and~\ref{thm:twoweightFS}.  For the convenience of the reader we restate both results here.

\begin{theorem*}(Theorem~\ref{thm:twoweightweaktype})
Let $0\leq\gamma<1$, $1\leq q<\infty$,
$0\leq\alpha<\frac{1}{q}$, $r=\frac{q}{1-\alpha q}$
and $(w,v) \in A^+_{q,r}(\gamma)$.
Then there is a constant $C=C(n,p,\gamma,q,r,\alpha)$ such that
\[
w^r( \{ M^{\gamma+}_{\alpha, c} f > \lambda \}) \leq [(w,v)]_{A_{q,r}^+(\gamma)}^r \frac{C}{\lambda^r} \biggl( \int_{\mathbb{R}^{n+1}} \lvert f \rvert^q \, v^q \biggr)^\frac{r}{q}
\]
for 
every $f\in L^1_{\mathrm{loc}}(\mathbb{R}^{n+1})$.

\end{theorem*}

\begin{proof}
For this proof we can make a number of reductions. 
First,  we may assume that $f$ is  bounded and has compact support. 
Second, it suffices to prove this result  for the following restricted maximal function
\[
M^{\gamma+}_{\alpha,\xi} f(x,t) = \sup_{\substack{ z(R^-(\gamma))=(x,t) \\ l(R)\geq \xi } } 
\lvert R^+(\gamma) \rvert^\alpha
\dashint_{R^+(\gamma)} \lvert f \rvert, \quad \xi<1,
\]
and then let $\xi \to 0$ to obtain the conclusion.
Third, it suffices to prove that
\begin{align*}
w^r( \{ \lambda < M^{\gamma+}_{\alpha, \xi} f \leq 2\lambda \}) \leq \frac{C}{\lambda^r} \biggl( \int_{\mathbb{R}^{n+1}} \lvert f \rvert^q \, v^q \biggr)^\frac{r}{q} ,
\end{align*}
since by summing up we get
\begin{align*}
w^r( \{ M^{\gamma+}_{\alpha,\xi} f > \lambda \}) &\leq
\sum_{k=0}^\infty w^r( \{ 2^{k} \lambda < M^{\gamma+}_{\alpha,\xi} f \leq 2^{k+1} \lambda \}) \\
&\leq
\sum_{k=0}^\infty
\frac{C}{2^{kr} \lambda^r} \biggl( \int_{\mathbb{R}^{n+1}} \lvert f \rvert^q \, v^q \biggr)^\frac{r}{q}
\leq
\frac{2C}{\lambda^r} \biggl( \int_{\mathbb{R}^{n+1}} \lvert f \rvert^q \, v^q \biggr)^\frac{r}{q} .
\end{align*}
Fourth, it suffices to prove that for every $L>0$ we have
\begin{align*}
w^r( B(0,L) \cap \{ \lambda < M^{\gamma+}_{\alpha,\xi} f \leq 2\lambda \}) \leq \frac{C}{\lambda^r} \biggl( \int_{\mathbb{R}^{n+1}} \lvert f \rvert^q \, v^q \biggr)^\frac{r}{q} ,
\end{align*}
since we may let $L\to\infty$ to obtain desired result.
Denote
\[
E = B(0,L) \cap \{ \lambda < M^{\gamma+}_{\alpha,\xi} f \leq 2\lambda \} .
\]
Finally, by Lemma~\ref{lem:trunctation},
it suffices to prove this result for the truncated weights $\max\{w,a\}$, $a>0$.
In other words,  we may assume that $w$ is bounded below by $a$.

By the definition of the set $E$,
for every point $z \in E$ there is a parabolic rectangle $R_{z}$ such that 
$z$ is the center point of $R^-_{z}(\gamma)$, $l(R_{z})\geq \xi$, and
\[
\lambda < \lvert R^+_{z}(\gamma) \rvert^{\alpha} \dashint_{R^+_{z}(\gamma)} \lvert f \rvert \leq 2 \lambda .
\]
Since $f\in L^1(\mathbb{R}^{n+1})$, we have that
\[
\lvert R^+_{z}(\gamma) \rvert^{1-\alpha} < \frac{1}{\lambda} \int_{R^+_{z}(\gamma)} \lvert f \rvert
\leq \frac{1}{\lambda} \int_{\mathbb{R}^{n+1}} \lvert f \rvert
< \infty;
\]
thus, the sidelength of $R_{z}$ is bounded  above.
Therefore, $\bigcup_{z\in E} R_{z}$ is a bounded set, and hence, by the absolute continuity of the integral,
there exists $0<\varepsilon<1$ such that for every $z\in E$,
\begin{align*}
w^r((1+\varepsilon)R_z^-(\gamma) \setminus R_z^-(\gamma)) \leq a (1-\gamma)\xi^{n+p} \leq a \lvert R_z^-(\gamma) \rvert \leq w(R_z^-(\gamma)).
\end{align*}
This also implies that
\[
w((1+\varepsilon)R_z^-(\gamma)) \leq 2 w(R_z^-(\gamma)) .
\]
Since $E$ is compact, we can find a finite collection of balls $B(z_k, (1-\gamma) \xi^p \varepsilon /2)$ with $z_k = (x_k,t_k) \in E$ such that
\[
E \subset \bigcup_{k} B(z_k, (1-\gamma) \xi^p \varepsilon /2) .
\]

We choose $R_{z_k}$ that has a largest sidelength and denote it by $R_1$. Then we consider a second largest $R_{z_k}$. If $z_k$ is contained in $R^-_1(\gamma)$, we discard $R_{z_k}$. Otherwise, we select $R_{z_k}$ and denote it by $R_{2}$.
Suppose that $R_j$, $j=1,\dots,i-1$, have been chosen. Now let  $R_{z_k}$ be the next largest rectangle (which may not be unique) and select it if $z_k$ is not contained in any of the previously chosen $R_j^-(\gamma)$ and denote it by $R_{i}$.
Therefore, after a finite number of steps we have constructed a finite collection $\{R_{i}\}_i$.
By construction, $R_i^-(\gamma) \nsubseteq R_j^-(\gamma) $ for $i\neq j$.
Hence,
every $z_k$ is contained in some $R^-_{i}(\gamma)$ and we have
$ B(z_k, (1-\gamma) \xi^p \varepsilon /2) \subset (1+\varepsilon)R^-_{i}(\gamma) $.
This implies that
\[
E \subset \bigcup_{k} B(z_k, (1-\gamma) \xi^p \varepsilon /2)
\subset \bigcup_{i} (1+\varepsilon) R^-_{i}(\gamma) .
\]
Therefore, we obtain
\begin{equation}
\label{eq:superlevelset}
w(E) \leq \sum_i w((1+\varepsilon) R^-_{i}(\gamma)) \leq 2 \sum_i w(R_i^-(\gamma)) .
\end{equation}

Denote the center of $R_i^-(\gamma)$ by $z_i$.
By construction, we observe that for $l(R_i)\geq l(R_j)$ we have
\begin{align*}
\lvert z_i - z_j \rvert \geq \frac{1}{2}
\min\{  l(R_i), (1-\gamma) l(R_i)^p \} .
\end{align*}
Thus, we have that $\frac{1}{2} R_i^-(\gamma) \cap \frac{1}{2} R_j^-(\gamma) = \emptyset $.
Fix $z\in\mathbb{R}^{n+1}$ and
define
\[
I_z^k = \{i\in\mathbb{N}: 2^{-k-1} < l(R_i) \leq 2^{-k}, z\in R_i^-(\gamma) \}, \quad k\in\mathbb{Z}.
\]
We then have that
\[
\bigcup_{i \in I_z^k} 
R_i^-(\gamma)
\subset 
S_z^-
,
\]
where 
$S_z^-$ is a rectangle centered at $z$ with sidelengths $l_x(S_z^-) = 2^{-k+1}$ and $l_t(S_z^-) = (1-\gamma)2^{-kp+1} $.
Hence, we have a bounded overlap for $R_i^-(\gamma)$ that have approximately same sidelength. In particular, 
\begin{align*}
\sum_{\substack{ i \\ 2^{-k-1} < l(R_i) \leq 2^{-k}}} 
\chi_{R_i^-(\gamma)}(z) 
&= 
\sum_{i\in I_z^k}
\frac{2^{n+1}}{(1-\gamma) l(R_i)^{n+p} } \Bigl\lvert \frac{1}{2} R_i^-(\gamma) \Bigr\rvert \\
&\leq
\frac{2^{n+1} 2^{(k+1)(n+p)}}{1-\gamma} 
\Bigl\lvert 
\bigcup_{i\in I_z^k}
R_i^-(\gamma) \Bigr\rvert \\
&\leq
\frac{2^{n+1} 2^{(k+1)(n+p)}}{1-\gamma} 
\lvert S_z^- \rvert \\
&=
\frac{2^{n+1} 2^{(k+1)(n+p)}}{1-\gamma } 
2^{(-k+1)n} (1-\gamma) 2^{-kp+1} \\
&= 2^{3n+p+2} .
\end{align*}

Fix $i$ and define
\[
J_i = \{j\in\mathbb{N}: R_i^+(\gamma) \cap R_j^+(\gamma) \neq \emptyset, l(R_j) < l(R_i) \} .
\]
We now divide the sets $J_i$ into two subcollections:
\[
J_i^1 = \{j\in J_i : R_j^+(\gamma) \nsubseteq R_i^+(\gamma) \} 
\quad\text{and}\quad
J_i^2 = \{j\in J_i : R_j^+(\gamma) \subset R_i^+(\gamma) \}.
\]
We first consider $J_i^1$.
Let $2^{-k_0-1}<l(R_i)\leq 2^{-k_0}$ and $2^{-k-1}<l(R_j)\leq 2^{-k}$.
If $R_j^+(\gamma) \nsubseteq R_i^+(\gamma)$, then
$R_j^+(\gamma)$ must intersect the boundary of $R_i^+(\gamma)$,  and so we have that
$R_j \subset A_k$ where $A_k$ 
is a set around the boundary of $R_i^+(\gamma)$ that satisfies
\begin{multline*}
    \lvert A_k \rvert \leq 
2\bigl(l(R_i) + 2l(R_i)\bigr)^n 2^{-kp+2}
+ 2n \bigl( l(R_i) + 2l(R_i) \bigr)^{n-1}
\bigl( (1-\gamma) l(R_i)^p + 4 l(R_i)^{p}\bigr)
2^{-k+1} 
\\
\leq 2^{2n+3} \bigl( l(R_i)^n 2^{-kp} + n l(R_i)^{n-1+p} 2^{-k} \bigr).
\end{multline*}
Therefore,
\begin{align*}
\sum_{j\in J_i^1} \lvert R_j^+(\gamma) \lvert
&= \sum_{k=k_0}^\infty \sum_{\substack{j\in J_i^1 \\ 2^{-k-1}<l(R_j)\leq 2^{-k} }} \lvert R_j^+(\gamma) \lvert\\
& =
\sum_{k=k_0}^\infty \sum_{\substack{j\in J_i^1 \\ 2^{-k-1}<l(R_j)\leq 2^{-k} }}
\lvert R_j^-(\gamma) \lvert \\
&= \sum_{k=k_0}^\infty 
\int_{A_k}
\sum_{\substack{j\in J_i^1 \\ 2^{-k-1}<l(R_j)\leq 2^{-k} }}
\chi_{R_j^-(\gamma)}(z) \dz 
\leq 
2^{3n+p+2} 
\sum_{k=k_0}^\infty
\lvert A_k \rvert \\
&\leq 2^{5n+p+5} \biggl( l(R_i)^n \sum_{k=k_0}^\infty 2^{-kp } + n l(R_i)^{n-1+p} \sum_{k=k_0}^\infty 2^{-k } \biggr) \\
&\leq 2^{5n+p+5} \biggl( l(R_i)^n 2^{-k_0 p+1} + n l(R_i)^{n-1+p} 2^{-k_0+1} \biggr) \\
&\leq 2^{5n+p+6} \biggl( l(R_i)^n 2^p l(R_i)^p + n l(R_i)^{n-1+p} 2 l(R_i) \biggr) \\
&\leq
\frac{C_1}{1-\gamma} \lvert R_i^+(\gamma) \rvert ,
\end{align*}
where $C_1 = 2^{6n+2p+6}$.

We now consider $J_i^2$.
Let 
\[
\mathcal{F}_{k_0} = \{ j\in J_i^2: 2^{-k_0-1}<l(R_j)\leq 2^{-k_0} \}
\]
and
\[
\mathcal{F}_{k} = \Bigl\{ j\in J_i^2: 2^{-k-1}<l(R_j)\leq 2^{-k}, R_j^-(\gamma) \cap \bigcup_{l=k_0}^{k-1} \bigcup_{m \in \mathcal{F}_l } R_m^-(\gamma) = \emptyset \Bigr\},
\quad k>k_0.
\]
In other words, we can view the rectangles in $\mathcal{F}_k$ as those that 
have  sidelength approximately $2^{-k}$ and do not intersect the previously chosen rectangles.
Now define $\mathcal{F} = \bigcup_{k=k_0}^\infty \mathcal{F}_k$.
Then for every $m\in J_i^2 \setminus \mathcal{F} $, there is $j\in \mathcal{F}$ such that $R_j^-(\gamma) \cap R_m^-(\gamma) \neq \emptyset$ and $l(R_m)< l(R_j)$.
Moreover, by the construction, we have that $R_m^-(\gamma) \nsubseteq R_j^-(\gamma) $.
Thus, we have
\begin{align*}
\sum_{j\in J_i^2} \lvert R_j^+(\gamma) \rvert \leq \sum_{j\in \mathcal{F}} \biggl( \lvert R_j^+(\gamma) \rvert +  
\sum_{m\in\mathcal{G}_j} \lvert R_m^+(\gamma) \rvert
\biggr) ,
\end{align*}
where
\[
\mathcal{G}_j = \{ m \in J_i^2 \setminus \mathcal{F}: R_j^-(\gamma) \cap R_m^-(\gamma) \neq \emptyset, R_m^-(\gamma) \nsubseteq R_j^-(\gamma) , l(R_m) < l(R_j) \} .
\]
We can now use an argument similar to that for $J_i^1$ on the second sum above 
by using the fact that the $R_j^-(\gamma)$ with approximately the same sidelength have bounded overlap.
More precisely,
let $2^{-k_0-1}<l(R_j)\leq 2^{-k_0}$ and $2^{-k-1}<l(R_m)\leq 2^{-k}$.
If $R_m^-(\gamma) \nsubseteq R_j^-(\gamma)$, then
$R_m^-(\gamma)$ intersects the boundary of $R_j^-(\gamma)$ and we have
$R_m^-(\gamma) \subset A_k$ where $A_k$ 
is a set around the boundary of $R_j^-(\gamma)$ such that
\begin{align*}
\lvert A_k \rvert 
&\leq
2\bigl( l(R_j) 
+ 2l(R_j) \bigr)^n (1-\gamma) 2^{-kp+1} \\
& \qquad \quad +2n \bigl( l(R_j) + 2l(R_j) \bigr)^{n-1}
\bigl( (1-\gamma) l(R_j)^{p} + 2(1-\gamma) l(R_j)^{p} \bigr)
2^{-k+1} \\
&\leq 
(1-\gamma) 2^{2n+2} \bigl( l(R_j)^n 2^{-kp} + n l(R_j)^{n-1+p} 2^{-k} \bigr) .
\end{align*}
Thus, we have
\begin{align*}
\sum_{m\in \mathcal{G}_j} \lvert R_m^+(\gamma) \rvert
&= \sum_{k=k_0}^\infty \sum_{\substack{m\in \mathcal{G}_j \\ 2^{-k-1}<l(R_m)\leq 2^{-k} }} \lvert R_m^+(\gamma) \rvert \\
& = \sum_{k=k_0}^\infty \sum_{\substack{m\in \mathcal{G}_j \\ 2^{-k-1}<l(R_m)\leq 2^{-k} }} \lvert R_m^-(\gamma) \rvert \\
&= \sum_{k=k_0}^\infty \int_{A_k}
\sum_{\substack{m\in \mathcal{G}_j \\ 2^{-k-1}<l(R_m)\leq 2^{-k} }}
\chi_{R_m^-(\gamma)}(z) \dz 
\leq 
2^{3n+p+2} \sum_{k=k_0}^\infty \lvert A_k \rvert \\
&\leq 
2^{5n+p+4} (1-\gamma) \biggl( l(R_j)^n \sum_{k=k_0}^\infty 2^{-kp} + n l(R_j)^{n-1+p} \sum_{k=k_0}^\infty 2^{-k} 
\biggr) \\
&\leq 2^{5n+p+4} (1-\gamma) \biggl( l(R_j)^n 2^{-k_0 p +1} + n l(R_j)^{n-1+p} 2^{-k_0+1} \biggr) \\
&\leq 2^{5n+p+5} (1-\gamma) \biggl( l(R_j)^n 2^p l(R_j)^p + n l(R_j)^{n-1+p} 2 l(R_j) \biggr) \\
&\leq
\frac{C_1}{2} \lvert R_j^+(\gamma) \rvert .
\end{align*}
Note that if $R_j^+(\gamma) \subset R_i^+(\gamma)$, then  $R_j^-(\gamma) \subset R_i$.
By our previous estimate, the definitions of $\mathcal{F}$ and $\mathcal{F}_k$, the bounded overlap of $R_j^-(\gamma)$ in $\mathcal{F}_k$, we get that
\begin{align*}
\sum_{j\in J_i^2} \lvert R_j^+(\gamma) \rvert 
&\leq \sum_{j\in \mathcal{F}} \biggl( \lvert R_j^+(\gamma) \rvert +  
\sum_{m \in \mathcal{G}_j} \lvert R_m^+(\gamma) \rvert
\biggr) \\
& \leq
\sum_{j\in \mathcal{F}} \biggl( \lvert R_j^+(\gamma) \rvert +  
\frac{C_1}{2} \lvert R_j^+(\gamma) \rvert
\biggr) \\
&\leq C_1 \sum_{k=k_0}^\infty \sum_{j\in \mathcal{F}_k} \lvert R_j^-(\gamma) \rvert \\
& = C_1 \sum_{k=k_0}^\infty \int_{\bigcup_{j'\in \mathcal{F}_k} R_{j'}^-(\gamma)} \sum_{j\in \mathcal{F}_k} \chi_{R_j^-(\gamma)}(z) \dz \\
&\leq C_1 2^{3n+p+2} \sum_{k=k_0}^\infty 
\Bigl\lvert \bigcup_{j\in \mathcal{F}_k} R_{j}^-(\gamma) \Bigr\rvert \\
& = C_1 2^{3n+p+2} \Bigl\lvert \bigcup_{k=k_0}^\infty \bigcup_{j\in \mathcal{F}_k} R_{j}^-(\gamma) \Bigr\rvert \\
&\leq
C_1 2^{3n+p+2} \Bigl\lvert \bigcup_{j\in J_i^2} R_{j}^-(\gamma) \Bigr\rvert
\leq C_1 2^{3n+p+2} \lvert R_i \rvert \\
& = \frac{C_1 2^{3n+p+3}}{1-\gamma} \lvert R_i^+(\gamma) \rvert .
\end{align*}

If we now combine the estimates for $J_i^1$ and $J_i^2$, we get
\begin{align*}
\sum_{j\in J_i} \lvert R_j^+(\gamma) \rvert &= \sum_{j\in J_i^1} \lvert R_j^+(\gamma) \rvert + \sum_{j\in J_i^2} \lvert R_j^+(\gamma) \rvert 
\leq \frac{C_1}{1-\gamma} \lvert R_i^+(\gamma) \rvert + \frac{C_1 2^{3n+p+3}}{1-\gamma} \lvert R_i^+(\gamma) \rvert \\
&\leq \frac{2^{9n+3p+10} }{1-\gamma} \lvert R_i^+(\gamma) \rvert
\leq
2^{-\frac{1}{1-\alpha}} c
\lvert R_i^+(\gamma) \rvert ,
\end{align*}
where $c = \lceil 2^{9n+3p+11+\frac{1}{1-\alpha}} / (1-\gamma) \rceil$.
Thus, we have that
\begin{align*}
\sum_{j\in J_i} \biggl( \int_{R_j^+(\gamma)} \lvert f \rvert \biggr)^\frac{1}{1-\alpha} \leq (2\lambda)^\frac{1}{1-\alpha} \sum_{j\in J_i} \lvert R_j^+(\gamma) \rvert \leq c \lambda^\frac{1}{1-\alpha} \lvert R_i^+(\gamma) \rvert
\leq c \biggl( \int_{R_i^+(\gamma)} \lvert f \rvert \biggr)^\frac{1}{1-\alpha} .
\end{align*}

Extract a collection of subsets $F_i$ from $\{R_i^+(\gamma)\}_i$  that have bounded overlap.  
Fix $i$ and denote the number of indices in $J_i$ by $N$.
If $N\leq 2^\frac{1}{1-\alpha} c$, we can choose $F_i = R_i^+(\gamma)$.
Otherwise, if $N>2^\frac{1}{1-\alpha} c$, we define
\[
G_i^k = \{z\in R_i^+(\gamma): \sum_{j\in J_i} \chi_{R_j^+(\gamma)}(z) \geq k \}, \quad k\in\mathbb{N}.
\]
For $k_1 < k_2$, we have $G_i^{k_2} \subset G_i^{k_1} $.
Moreover, observe that 
\begin{align*}
\sum_{k=1}^N \chi_{G_i^k}(z) = \sum_{j\in J_i} \chi_{R_j^+(\gamma)}(z) .
\end{align*}
Therefore, we have that
\begin{align*}
N^{\frac{\alpha}{1-\alpha} } 2^\frac{1}{1-\alpha} c  \biggl( \int_{G_i^{2c}} \lvert f \rvert \biggr)^\frac{1}{1-\alpha} 
&\leq 
\biggl( N \int_{G_i^{2c}} \lvert f \rvert \biggr)^\frac{1}{1-\alpha}
\leq
\biggl( \sum_{k=1}^N \int_{G_i^k} \lvert f \rvert \biggr)^\frac{1}{1-\alpha}
= \biggl( \int_{R_i^+(\gamma)} \lvert f \rvert \sum_{k=1}^N \chi_{G_i^k} \biggr)^\frac{1}{1-\alpha} \\
&= \biggl( \int_{R_i^+(\gamma)} \lvert f \rvert \sum_{j\in J_i} \chi_{R_j^+(\gamma)} \biggr)^\frac{1}{1-\alpha}
\leq \biggl( \sum_{j\in J_i} \int_{R_j^+(\gamma)} \lvert f \rvert \biggr)^\frac{1}{1-\alpha} \\
&\leq N^{\frac{\alpha}{1-\alpha} } \sum_{j\in J_i} \biggl( \int_{R_j^+(\gamma)} \lvert f \rvert \biggr)^\frac{1}{1-\alpha}
\leq N^{\frac{\alpha}{1-\alpha} } c \biggl( \int_{R_i^+(\gamma)} \lvert f \rvert \biggr)^\frac{1}{1-\alpha} .
\end{align*}
This implies that
\[
\int_{G_i^{2c}} \lvert f \rvert 
\leq
\frac{1}{2} \int_{R_i^+(\gamma)} \lvert f \rvert .
\]

Define $F_i = R_i^+(\gamma) \setminus G_i^{2c}$; then we have that
\begin{align*}
\int_{F_i} \lvert f \rvert = \int_{R_i^+(\gamma)} \lvert f \rvert - \int_{G_i^{2c}} \lvert f \rvert \geq \frac{1}{2} \int_{R_i^+(\gamma)} \lvert f \rvert > \frac{\lambda}{2} \lvert R_i^+(\gamma) \rvert^{1-\alpha} .
\end{align*}
Thus, for every $F_i$ we have that
\begin{equation}
\label{eq:intavgFi}
\frac{2}{\lvert R_i^+(\gamma) \rvert^{1-\alpha}} \int_{F_i} \lvert f \rvert > \lambda .
\end{equation}
Moreover, for every $z\in F_{i}$ we have 
\begin{equation}
\label{eq:bddoverlapFi}
\sum_{j\in J_{i}} \chi_{R_j^+(\gamma)}(z) \leq 2^\frac{1}{1-\alpha} c .
\end{equation}

Now observe that the $R_i^+(\gamma)$ which have approximately the same sidelength  must also have bounded overlap. 
In particular, we can follow the argument for $R_i^-(\gamma)$ to show the following:
Fix $z\in\mathbb{R}^{n+1}$ and denote
\[
I_z^k = \{i\in\mathbb{N}: 2^{-k-1} < l(R_i) \leq 2^{-k}, z\in R_i^+(\gamma) \}, \quad k\in\mathbb{Z}.
\]
Then we have
\[
\bigcup_{i \in I_z^k} 
R_i^-(\gamma)
\subset 
S_z
,
\]
where 
$S_z$ is a rectangle centered at $z$ with sidelengths $l_x(S_z) = 2^{-k+1}$ and $l_t(S_z) = 2^{-kp+1} $.
Thus, we obtain
\begin{align*}
\sum_{\substack{ i \\ 2^{-k-1} < l(R_i) \leq 2^{-k}}} 
\chi_{R_i^+(\gamma)}(z) 
&= 
\sum_{i\in I_z^k}
\frac{2^{n+1}}{(1-\gamma) l(R_i)^{n+p} } \Bigl\lvert \frac{1}{2} R_i^-(\gamma) \Bigr\rvert 
\leq
\frac{2^{n+1} 2^{(k+1)(n+p)}}{1-\gamma } 
\Bigl\lvert 
\bigcup_{i\in I_z^k}
R_i^-(\gamma) \Bigr\rvert \\
&\leq
\frac{2^{n+1} 2^{(k+1)(n+p)}}{1-\gamma } 
\lvert S_z \rvert 
=
\frac{2^{n+1} 2^{(k+1)(n+p)}}{1-\gamma } 
2^{(-k+1)n} 2^{-kp+1}
= \frac{2^{3n+p+2}}{1-\gamma}  .
\end{align*}
We now claim that the $F_i$ have bounded overlap.
Fix $z\in \bigcup_i F_i$ and fix a rectangle $R_{i_0} \in \{R_i\}_i$ with largest sidelength such that $z\in F_{i_0}$.
Let $2^{-k_0-1}< l(R_{i_0}) \leq 2^{-k_0}$.
By the previous estimate and~\eqref{eq:bddoverlapFi}, we have
\begin{align*}
\sum_i \chi_{F_i}(z) 
&= \sum_{\substack{i \\ 2^{-k_0-1}< l(R_i) \leq 2^{-k_0} }} \chi_{F_i}(z) 
+ \sum_{\substack{i \\ l(R_i)\leq 2^{-k_0-1} }} \chi_{F_i}(z) \\
&\leq \sum_{\substack{i \\ 2^{-k_0-1}< l(R_i)\leq 2^{-k_0} }} \chi_{R_i^+(\gamma)}(z)
+ \sum_{i \in J_{i_0}} \chi_{R_i^+(\gamma)}(z) \\
&\leq 
\frac{2^{3n+p+2}}{1-\gamma} + 2^\frac{1}{1-\alpha} c \\
& \leq C_2 ,
\end{align*}
where $C_2 = 2^{9n+3p+13 +\frac{1}{1-\alpha} } /(1-\gamma) $.

We first suppose that  $1<q<\infty$.
By~\eqref{eq:superlevelset},~\eqref{eq:intavgFi}, H\"older's inequality, $\alpha=\frac{1}{q}-\frac{1}{r}$, and the bounded overlap of the $F_i$, we conclude that
\begin{align*}
w^r(E) &\leq 
2 \sum_i w^r(R_i^-(\gamma))
\leq \frac{2^{r+1}}{\lambda^r} \sum_i w^r(R_i^-(\gamma)) \biggl( \frac{1}{\lvert R_i^+(\gamma) \rvert^{1-\alpha} } \int_{F_i} \lvert f \rvert \biggr)^r \\
&\leq \frac{2^{r+1}}{\lambda^r} \sum_i  \frac{w^r(R_i^-(\gamma))}{\lvert R_i^+(\gamma) \rvert^{r(1-\alpha)} } \biggl( \int_{F_i} v^{-q'} \biggr)^{ \frac{r}{q'} } \biggl( \int_{F_i} \lvert f \rvert^q \, v^q \biggr)^\frac{r}{q} \\
&\leq \frac{2^{r+1}}{\lambda^r} \sum_i \dashint_{R_i^-(\gamma)} w^r  \biggl( \dashint_{R_i^+(\gamma)} v^{-q' } \biggr)^{ \frac{r}{q'} } \biggl( \int_{F_i} \lvert f \rvert^q \, v^q \biggr)^\frac{r}{q} \\
&\leq \frac{2^{r+1}}{\lambda^r} [(w,v)]_{A_{q,r}^+(\gamma)}^r \sum_i \biggl(\int_{F_i} \lvert f \rvert^q \, v^q \biggr)^\frac{r}{q} \\
&\leq \frac{2^{r+1}}{\lambda^r} [(w,v)]_{A_{q,r}^+(\gamma)}^r \biggl( \sum_i  \int_{F_i} \lvert f \rvert^q \, v^q \biggr)^\frac{r}{q} \\
&\leq \frac{2^{r+1} C_2^\frac{r}{q}}{\lambda^r} [(w,v)]_{A_{q,r}^+(\gamma)}^r \biggl(  \int_{\mathbb{R}^{n+1}} \lvert f \rvert^q \, v^q \biggr)^\frac{r}{q} .
\end{align*}

When $q=1$, we use~\eqref{eq:superlevelset},~\eqref{eq:intavgFi}, and the bounded overlap of $F_i$ to get
\begin{align*}
w^r(E) &\leq 
2 \sum_i w^r(R_i^-(\gamma))
\leq \frac{2^{r+1}}{\lambda^r} \sum_i w^r(R_i^-(\gamma)) \biggl( \frac{1}{\lvert R_i^+(\gamma) \rvert^{1-\alpha}} \int_{F_i} \lvert f \rvert \biggr)^r
\\
&=
\frac{2^{r+1}}{\lambda^r} \sum_i \dashint_{R_i^-(\gamma)} w^r
\biggl( \int_{F_i} \lvert f \rvert \biggr)^r
\leq \frac{2^{r+1}}{\lambda^r} [(w,v)]_{A_{1,r}^+(\gamma)}^r \sum_i \biggl( \int_{F_i} \lvert f \rvert \, v \biggr)^r 
\\
&\leq 
\frac{2^{r+1}}{\lambda^r} [(w,v)]_{A_{1,r}^+(\gamma)}^r \biggl( \sum_i  \int_{F_i}  \lvert f \rvert \, v \biggr)^r
\leq \frac{2^{r+1} C_2^r}{\lambda^r} [(w,v)]_{A_{1,r}^+(\gamma)}^r \biggl( \int_{\mathbb{R}^{n+1}} \lvert f \rvert \, v \biggr)^r .
\end{align*}
This completes the proof.
\end{proof}

\begin{theorem*}(Theorem~\ref{thm:twoweightFS})
Let $0\leq\gamma<1$, $1<q<\infty$
and let $w$ be a locally integrable weight.
Then there are constants $C_1 = C_1(n,p,\gamma)$ and $C_2=C_2(n,p,\gamma,q)$ such that
\[
w( \{ M_{c}^{\gamma+}f > \lambda \}) \leq \frac{C_1}{\lambda} \int_{\mathbb{R}^{n+1}} \lvert f \rvert \, M^{\gamma-}w\,dx\dt
\]
for 
every $f\in L^1_{\mathrm{loc}}(\mathbb{R}^{n+1})$.
Moreover, 
\[
\int_{\mathbb{R}^{n+1}} (M_c^{\gamma+}f)^q \, w \,dx\dt \leq C_2  \int_{\mathbb{R}^{n+1}} \lvert f \rvert^q \, M^{\gamma-}w\,dx\dt
\]
for 
every $f\in L^1_{\mathrm{loc}}(\mathbb{R}^{n+1})$.
\end{theorem*}

\begin{proof}
By Proposition~\ref{thm:maximalA1-cond}, we have $(w, M^{\gamma-}w)\in A_{1,1}^+(\gamma)$.
Hence,
Theorem~\ref{thm:twoweightweaktype} implies
the first inequality of the claim.
We note that
\[
\norm{M^{\gamma+}_c f}_{L^\infty(w)} \leq \norm{f}_{L^\infty(M^{\gamma-}w)} .
\]
Thus, $M^{\gamma+}_c f$ is bounded from $L^1(M^{\gamma-}w)$ to $L^{1,\infty}(w)$ and from $L^\infty(M^{\gamma-}w)$ to $L^{\infty}(w)$.
Therefore, by the Marcinkiewicz interpolation theorem, $M^{\gamma+}_c f$ is bounded from $L^q(M^{\gamma-}w)$ to $L^{q}(w)$: more precisely,
\[
\int_{\mathbb{R}^{n+1}} (M_c^{\gamma+}f)^q \, w \leq
\frac{q 2^{q+1} C_1  }{q-1}  
\int_{\mathbb{R}^{n+1}} \lvert f \rvert^q \, M^{\gamma-}w .
\]
\end{proof}

\section{strong-type estimates}
\label{section:strong}

In this section, we prove Theorems~\ref{thm:twoweightstrongbump} and~\ref{thm:twoweightstrongsawyer}.  Again, for the convenience of the reader we repeat the statement of both results.

\begin{theorem*} (Theorem~\ref{thm:twoweightstrongbump})
Let $0\leq\gamma<1$, $1<q,\,s<\infty$ and $w,v$ to be nonnegative measurable functions.
Assume that 
\[
[(w,v)]_{A^+_{q,q,s}(\gamma)} = 
\sup_{R \subset \mathbb R^{n+1}} \biggl( \dashint_{R^-(\gamma)} w  \biggr) \biggl( \dashint_{R^+(\gamma)} v^{\frac{s}{1-q}} \biggr)^\frac{q-1}{s} < \infty .
\]
Then there is a constant $C=C(n,p,\gamma,q,s)$ such that
\[
\int_{\mathbb{R}^{n+1}} (M_c^{\gamma+}f)^q \, w
\leq C [(w,v)]_{A_{q,q,s}^+(\gamma)} \int_{\mathbb{R}^{n+1}} \lvert f \rvert^q \, v 
\]
for every $f\in L^1_{\mathrm{loc}}(\mathbb{R}^{n+1})$.
\end{theorem*}

\begin{proof}
Let $t$ be such that $t-1 = \frac{q-1}{s}$. Note that $1<t<q$. Then $[(w^\frac{1}{t},v^\frac{1}{t})]_{A_{t,t}^+(\gamma)}^t = [(w,v)]_{A_{q,q,s}^+(\gamma)}$ since
\[
\biggl( \dashint_{R^-(\gamma)} w \biggr) \biggl( \dashint_{R^+(\gamma)} v^{\frac{1}{1-t}} \biggr)^{t-1}
=
\biggl( \dashint_{R^-(\gamma)} w \biggr) \biggl( \dashint_{R^+(\gamma)} v^{\frac{s}{1-q}} \biggr)^\frac{q-1}{s}
\]
for every parabolic rectangle $R \subset \mathbb R^{n+1}$.
Thus, $(w^\frac{1}{t} ,v^\frac{1}{t} )\in A^+_{t,t}(\gamma)$; hence, by Theorem~\ref{thm:twoweightweaktype} we have that
\[
w( \{ M^{\gamma+}_c f > \lambda \}) 
\leq [(w^\frac{1}{t},v^\frac{1}{t})]_{A_{t,t}^+(\gamma)}^t \frac{C}{\lambda^t} \int_{\mathbb{R}^{n+1}} \lvert f \rvert^t \, v .
\]
Moreover, observe that
\[
\norm{M^{\gamma+}_c f}_{L^\infty(w)} \leq \norm{f}_{L^\infty(v)} .
\]
In other words, $M_c^{\gamma+}f$ is bounded from $L^t(v)$ to $L^{t,\infty}(w)$ and from $L^\infty(v)$ to $L^{\infty}(w)$.
Therefore, by the  Marcinkiewicz interpolation theorem, $M_c^{\gamma+}f$ is bounded from $L^q(v)$ to $L^{q}(w)$, i.e.,
\begin{align*}
\int_{\mathbb{R}^{n+1}} (M_c^{\gamma+}f)^q \, w &\leq
 \frac{ q}{q-t}  2^q  C [(w^\frac{1}{t},v^\frac{1}{t})]_{A_{t,t}^+(\gamma)}^t
\int_{\mathbb{R}^{n+1}} \lvert f \rvert^q \, v
\\
&=
\frac{q}{q-1} \frac{s}{s-1} 2^q  C [(w,v)]_{A_{q,q,s}^+(\gamma)}
\int_{\mathbb{R}^{n+1}} \lvert f \rvert^q \, v .
\end{align*}
\end{proof}

\begin{theorem*}(Theorem~\ref{thm:twoweightstrongsawyer})
Let $1<q\leq r<\infty$, $0\leq\alpha<1$
and $w,\,v$ to be weights.
Then 
\[
[(w,v)]_{S_{q,r,\alpha}^+} = 
\sup_R \biggl(  \int_{R^+} v^{1-q'} \biggr)^{-\frac{1}{q}}
\biggl( \int_{R} (M^{+}_{\alpha}(v^{1-q'} \chi_{R^+}))^r \, w \biggr)^\frac{1}{r} < \infty
\]
for every parabolic rectangle $R\subset\mathbb{R}^{n+1}$ if and and only if 
there is a constant $C=C(n,p,q,r)$ such that
\[
\biggl( \int_{\mathbb{R}^{n+1}} (M_{\alpha,c}^{+}f)^r \, w \biggr)^\frac{1}{r} 
\leq
C [(w,v)]_{S_{q,r,\alpha}^+} \biggl( \int_{\mathbb{R}^{n+1}} \lvert f \rvert^q \, v \biggr)^\frac{1}{q}
\]
for 
every $f\in L^1_{\mathrm{loc}}(\mathbb{R}^{n+1})$.
\end{theorem*}

\begin{proof}
First, it is immediate that the norm inequality implies the testing condition $S_{q,r,\alpha}^+$, as it is just the norm inequality for the family of functions $v^{1-q'}\chi_{R^+}$.  

To prove sufficiency,  we must first develop the machinery of dyadic grids in the parabolic geometry.  We start with a dyadic grid in the spatial coordinates.  
In $\mathbb{R}^{n}$ let $\mathcal{D} = \bigcup_{k\in\mathbb{Z}} \mathcal{D}_k$, be an abstract dyadic grid in $\R^n$:  that is,
 each $\mathcal{D}_k$ consists of pairwise disjoint cubes $Q$ with sidelength $l(Q)=2^{-k}$ such that $\bigcup_{Q\in\mathcal{D}_k} Q = \mathbb{R}^n$.  Moreover, given cubes $Q,\,P \in \mathcal{D}$, $Q\cap P \in \{ Q, P, \emptyset\}$.  We will always assume that $k=0$ is such that $l(Q)=1$ for all $Q\in \mathcal{D}_0$. 
 
We now extend these cubes to parabolic rectangles in $\R^{n+1}$.
Divide the time coordinate $\mathbb{R}$ into pairwise disjoint intervals $I_0$ with sidelength $\frac{1}{2}$ such that they cover $\mathbb{R}$.
Define rectangles $S_0^- = Q \times I_0$ for every $Q\in\mathcal{D}_0$ and for every $I_0$. Then the $S_0^-$ are pairwise disjoint, cover $\mathbb{R}^{n+1}$, and are such that $l_x(S_0^-)=1$ and $l_t(S_0^-)=\tfrac{1}{2}$.
We denote the collection of $S_0^-$ by $\mathcal{R}_0$
and call it the dyadic parabolic lattice at level $0$.

We now construct dyadic parabolic lattices at every level $k\in\mathbb{Z}$.
We first consider $k\in\mathbb{Z}_+$.
To partition each $S^-_0\in \mathcal{R}_0$, we divide each spatial edge into $2$ intervals of equal length. (Equivalently, we divide the spatial base in $\mathcal{D}_0$ into its $2^n$ dyadic children.) 
If
\[
\frac{1}{2 \lfloor 2^{p} \rfloor}=\frac{l_t(S_{0}^-)}{\lfloor 2^{p} \rfloor} < \frac{1}{2^{p+1}} ,
\]
then we divide the time interval of $S^-_0$ into $\lfloor 2^{p} \rfloor$ intervals of equal length.
Otherwise, we divide the time interval of $S^-_0$ into $\lceil 2^{p} \rceil$ intervals of equal length.
In either case we get subrectangles $S^-_1$ of  $S^-_0$ whose spatial sidelength 
$l_x(S^-_1)=l_x(S^-_0)/2 = 2^{-1}$ and whose time length is either 
\[
l_t(S^-_1)=\frac{l_t(S^-_0)}{\lfloor 2^{p} \rfloor} 
=\frac{1}{2\lfloor 2^{p} \rfloor} 
\quad\text{or}\quad
l_t(S^-_1)=\frac{1}{2\lceil 2^{p} \rceil}.
\]
We subdivide $S^-_1$ in the same manner as above to obtain $S^-_2$.

We now repeat this process by induction.
At the $k$th step,  partition the
rectangles $S^-_{k-1}$ by dividing each spatial side into $2$  intervals of equal length.
If 
\begin{equation}
\label{JNproof_eq1}
\frac{l_t(S_{k-1}^-)}{\lfloor 2^{p} \rfloor} < \frac{1}{2^{kp+1}},
\end{equation}
then divide the time interval of $S^-_{k-1}$ into $\lfloor 2^{p} \rfloor$ intervals of equal length.
Otherwise, if
\begin{equation}
\label{JNproof_eq2}
\frac{l_t(S_{k-1}^-)}{\lfloor 2^{p} \rfloor} \geq \frac{1}{2^{kp+1}},
\end{equation}
divide the time interval of $S^-_{k-1}$ into $\lceil 2^{p} \rceil$ intervals of equal length.
Thus, we get pairwise disjoint subrectangles $S^-_k$
with spatial sidelength 
$l_x(S^-_k) = 2^{-k}$. For each $k\in\mathbb{Z}_+$ define the  dyadic parabolic lattice $\mathcal{R}_k$ at level $k$ to be the collection of all the $S^-_k$ we have constructed.

For all $k\geq 0$ and  $S_k^-\in \mathcal{R}_k$, we claim that for the time length we have that 
\begin{equation}
\label{eq:sidelengths}
\frac{1}{2^{kp+2}} \leq l_t(S^-_k) \leq \frac{1}{2^{kp+1}}.
\end{equation}
We prove this by induction.  If $k=0$, then for each $S_0^- \in \mathcal{R}_0$, $l_t(S_0^-) =\frac{1}{2}=\frac{1}{2^{0p+1}}\geq \frac{1}{2^{0p+2}}$.  
Now suppose that for some $k\in \mathbb{N}$, \eqref{eq:sidelengths} holds for all $S_{k-1}^- \in \mathcal{R}_{k-1}$.  Fix $S_k^- \in \mathcal{R}_k$ with parent $S_{k-1}^-$.  If~\eqref{JNproof_eq1} holds, then we have that
\[ l_t(S_k^-) = \frac{l_t(S_{k-1}^-)}{\lfloor 2^{p} \rfloor} < \frac{1}{2^{kp+1}}.  \]
On the other hand, by induction, 
\[ l_t(S_k^-) = \frac{l_t(S_{k-1}^-)}{\lfloor 2^{p} \rfloor} \geq \frac{1}{2^{(k-1)p+2}}\frac{1}{2^p} = \frac{1}{2^{kp+2}}. \]
If~\eqref{JNproof_eq2} holds, then we have that
\[ l_t(S_k^-) = \frac{l_t(S_{k-1}^-)}{\lceil 2^{p} \rceil} \geq \frac{\lfloor 2^{p} \rfloor}{2^{(k-1)p+1}}\frac{1}{\lceil 2^{p} \rceil}
\geq \frac{1}{2^{(k-1)p+2}} \geq \frac{1}{2^{kp+2}}. \]
And, again by induction, 
\[ l_t(S_k^-) = \frac{l_t(S_{k-1}^-)}{\lceil 2^{p} \rceil} \leq \frac{1}{2^{(k-1)p+1}}\frac{1}{2^p} = \frac{1}{2^{kp+1}}. \]
Thus, \eqref{eq:sidelengths} holds for all $k\geq 0$.

We now consider the case $k\in\mathbb{Z}_-$.  We proceed by essentially reversing the previous construction.  Take any cube $Q\in \mathcal{D}_{-1}$ with sidelength 2, and erect on it a parabolic rectangle $S_{-1}^k$ whose sidelength is either $\lfloor 2^p \rfloor/2$ or $\lceil 2^p \rceil/2$ in such a way that \eqref{eq:sidelengths} holds, and such that if the construction used above starting from \eqref{JNproof_eq1} and \eqref{JNproof_eq2}, is applied to it, then the resulting parabolic rectangles $S_0^{-}$ are all contained in $\mathcal{R}_0$.  Denote the union of all the parabolic rectangles constructed in this way by $\mathcal{R}_{-1}$.  Continue this construction inductively:  given $k \in \mathbb{N}$, given the parabolic rectangles in $\mathcal{R}_{-k+1}$, on each cube in $\mathcal{D}_{-k}$ construct a parabolic rectangle whose time length is either $\lfloor 2^p \rfloor$ or $\lceil 2^p \rceil$ times the time length of the rectangles in $\mathcal{R}{_{-k+1}}$, \eqref{eq:sidelengths} holds, and such that if the construction used above starting from \eqref{JNproof_eq1} and \eqref{JNproof_eq2}, is applied to it, then the resulting parabolic rectangles $S_{-k+1}^{-}$ are all contained in $\mathcal{R}_{-k+1}$. Denote these parabolic rectangles by $\mathcal{R}_{-k}$.  Since we can repeat this for all $k$, 
we have constructed the desired dyadic structure $\mathcal{R}=\bigcup_{k\in\mathbb{Z}} \mathcal{R}_k$ in the parabolic geometry.

For every $S^-\in\mathcal{R}_k$ there exists a unique dyadic parabolic rectangle $\widetilde{S}^-\in\mathcal{R}_{k-1}$ such that $S^- \subset \widetilde{S}^-$.  Denote by $S^+_k$ the rectangle with the same sidelengths as $S^-_k$ 
such that the bottom of $S^+_k$ coincides with the top of $S^-_k$.
Note that by our construction, the collection of $S^+_k$ also forms $\mathcal{R}_k$.
For every $S^+_k$ there exists a unique rectangle $R^+_k$ with spatial sidelength $l_x(R_k^+)=2^{-k}$ and time length $l_t(R_k^+)=2^{-kp}$ such that $R^+_k$ has the same bottom as $S^+_k$.
By \eqref{eq:sidelengths}, we have that
\begin{multline*}
l_t(R^+_k) - l_t(S^+_k) 
\leq \frac{1}{2^{kp}} - \frac{1}{2^{kp+2}} 
= \frac{3}{4} \frac{1}{2^{kp}} 
\leq \frac{1}{2^{kp}} 
\\
\leq \frac{1}{2^{(k-1)p+1}} 
= \frac{1}{2^{(k-1)p}} - \frac{1}{2^{(k-1)p+1}} 
\leq l_t(R^+_{k-1}) - l_t(S^+_{k-1}) .
\end{multline*}
It follows at once from this that
\begin{equation}
\label{eq:subset}
R^+_{k} \subset R^+_{k-1}
\end{equation}
for a fixed rectangle $S^+_{k-1}$ and for every subrectangle $S^+_{k} \subset S^+_{k-1}$.

\medskip

We will show that translates of these grids in both the spatial and time coordinates will yield a finite family of parabolic dyadic grids that satisfies the analog of the three lattice theorem.  
Let $R \subset\mathbb{R}^{n+1}$ be an arbitrary parabolic rectangle
with
$R^+ = Q\times I$.
By the three lattice theorem (see~\cite[Theorem~3.1]{lernernazarov2019} or~\cite[proof of Theorem~1.10]{MR3092729}) 
there exist $3^n$ dyadic lattices $\mathcal{D}^\iota$, $\iota=1,\dots,3^n$, such that
for any cube $Q\subset\mathbb{R}^n$ there is a dyadic cube
$Q'\in \mathcal{D}^\iota$ for some $\iota\in\{1,\dots,3^n\}$ 
with
$Q\subset Q'$ and $l(Q') \leq 3l(Q)$.
Denote $L=l(Q)$.
Let $Q_k$ be the unique dyadic cube for which 
$Q' \subset Q_k$ and $l(Q_k)=2^{-k} = 2^3 l(Q')$.
It holds that $2^3 L\leq 2^{-k} \leq 2^5 L$.
Consider corresponding dyadic parabolic rectangles $S^+_k$ which spatial projection is $Q_k$.
Then by \eqref{eq:sidelengths} 
we have that  the time length of $S^+_k$ satisfies
\[
2 L^p \leq 2^{3p-2} L^p \leq 2^{-kp-2} \leq l_t(S^+_k) < 2^{-kp-1} \leq 2^{5p-1} L^p .
\]
It follows from this that if we take $l_t(S^+_k) / (\frac{1}{2}L^p) \leq \lceil 2^{5p} \rceil$ time translations of our lattice $\mathcal{R}_k$, with steps of length approximately $L^p/2$, then there would exist a dyadic parabolic rectangle $S^+_k$
in one of these lattices such that 
$R^+\in S_k^+$ and such that $z(R^-)\in S^-_k$.

Therefore, we have shown that there exist $3^n \lceil 2^{5p} \rceil$ dyadic parabolic 
lattices
$\mathcal{R}^\iota$, $\iota=1,\dots, 3^n \lceil 2^{5p} \rceil$, that our translations (in space and time coordinates) of our constructed lattice $\mathcal{R}$,
such that
for 
any
parabolic rectangle $R$ there is $S^+_k\in \mathcal{R}^\iota$ for some $\iota\in \{1,\dots, 3^n \lceil 2^{5p} \rceil\}$ 
with
$R^+\subset S^+_k$, $z(R^-)\in S^-_k$, 
\[
2^3 l_x(R^+) \leq l_x(S^+_k) < 2^5 l_x(R^+)
\quad\text{and}\quad
2^{3p-2} l_t(R^+) \leq l_t(S^+_k) < 2^{5p-1} l_t(R^+) .
\]
Furthermore, we have
\begin{align*}
\biggl( \lvert R^+ \rvert^\alpha \dashint_{R^+} \lvert f \vert \biggr)^r \leq 2^{r(1-\alpha)(5n+5p-1)} \biggl( \lvert S^+_k \rvert^\alpha \dashint_{S^+_k} \lvert f \vert \biggr)^r
\end{align*}
with $z(R^-) \in S^-_k$.
It follows that
\begin{align*}
\bigl( M^{+}_{\alpha,c} f(x,t) \bigr)^r \leq C_1^r \sum_{\iota=1}^{3^n \lceil 2^{5p} \rceil } \bigl( M^{+}_{\alpha,d,\iota} f(x,t) \bigr)^r
\end{align*}
for every $(x,t)\in\mathbb{R}^{n+1}$,
where
\[
M_{\alpha,d,\iota}^{+}f(x,t) = \sup_{\substack{ S^-\ni(x,t) \\ S^-\in\mathcal{R}^\iota }} \lvert S^+ \rvert^\alpha \dashint_{S^+} \lvert f \rvert
\]
is the dyadic parabolic maximal function with respect to the 
lattice $\mathcal{R}^\iota$ 
and $C_1 = 2^{(1-\alpha)(5n+5p-1)}$.

Therefore, to complete the proof it will suffice to prove that
\[
\biggl(
\int_{\mathbb{R}^{n+1}} (M_{\alpha,d,\iota}^{+}f)^r \, w \biggr)^\frac{1}{r} 
\leq
C [(w,v)]_{S_{q,r,\alpha}^+} 
\biggl(
\int_{\mathbb{R}^{n+1}} \lvert f \rvert^q \, v \biggr)^\frac{1}{q}
\]
for an arbitrary dyadic parabolic lattice $\mathcal{R}^\iota$,
since then 
\begin{multline*}
\biggl( \int_{\mathbb{R}^{n+1}} (M_{\alpha,c}^{+}f)^r \, w \biggr)^\frac{1}{r}  \leq
\biggl( 
\int_{\mathbb{R}^{n+1}} C_1^r  \sum_{\iota=1}^{3^n \lceil 2^{5p} \rceil } ( M_{\alpha,d,\iota}^{+}f)^r \, w \biggr)^\frac{1}{r}  \\
\leq
C_1 \sum_{\iota=1}^{3^n \lceil 2^{5p} \rceil } 
\biggl( 
\int_{\mathbb{R}^{n+1}} ( M_{\alpha,d,\iota}^{+}f)^r \, w \biggr)^\frac{1}{r}  
\leq
C_1 3^n \lceil 2^{5p} \rceil C [(w,v)]_{S_{q,r,\alpha}^+} \biggl( \int_{\mathbb{R}^{n+1}} \lvert f \rvert^q \, v \biggr)^\frac{1}{q} .
\end{multline*}
Furthermore, by a standard approximation argument we may assume that $f$ is a bounded function of compact support.

Fix a dyadic parabolic lattice $\mathcal{R}^\iota$ and denote the corresponding dyadic parabolic maximal function by $M_{\alpha,d}^{+} = M_{\alpha,d,\iota}^{+}$.
Let $k\in\mathbb{Z}$.
Then, we can use the properties of the dyadic lattice to argue as in the classical Calder\'on--Zygmund decomposition, and show that there  exists a collection $\{S_{i,k}^-\}_i \subset \mathcal{R}^\iota$ of pairwise disjoint dyadic parabolic rectangles such that
\[
\{ M_{\alpha,d}^{+}f > 2^k \} = \bigcup_i S_{i,k}^-
\quad \text{and} \quad
\lvert S_{i,k}^+ \rvert^\alpha \dashint_{S_{i,k}^+} \lvert f \rvert > 2^k .
\]
Define
\[
F_{i,k} = S_{i,k}^- \cap \{ M_{\alpha,d}^{+}f \leq 2^{k+1} \};
\]
then $\{F_{i,k}\}_{i,k}$ is a collection of pairwise disjoint sets.
Let $\sigma = v^{1-q'}$.
By \eqref{eq:sidelengths}, we have that
\begin{align*}
\int_{\mathbb{R}^{n+1}} (M_{\alpha,d}^{+}f)^r w
&\leq
\sum_k 2^{(k+1)r} \int_{ \{ 2^k < M_{\alpha,d}^{+}f \leq 2^{k+1} \} } w \\
& =  2^r \sum_{i,k} 2^{kr} w(F_{i,k}) \\
&\leq 2^r \sum_{i,k} \biggl( \lvert S_{i,k}^+ \rvert^\alpha \dashint_{S_{i,k}^+} \lvert f \rvert \biggr)^r w(F_{i,k}) \\
&= 2^r \sum_{i,k} \biggl( \frac{1}{\sigma(R_{i,k}^+)} \int_{S_{i,k}^+} \lvert f \rvert \biggr)^r \Bigl( \lvert R_{i,k}^+ \rvert^\alpha \sigma_{R_{i,k}^+} \Bigr)^r
\frac{\lvert R_{i,k}^+ \rvert^{r(1-\alpha)} }{\lvert S_{i,k}^+ \rvert^{r(1-\alpha)} }
w(F_{i,k}) \\
&\leq 8^r \sum_{i,k} \biggl( \frac{1}{\sigma(R_{i,k}^+)} \int_{S_{i,k}^+} \lvert f \rvert \biggr)^r \Bigl( \lvert R_{i,k}^+ \rvert^\alpha \sigma_{R_{i,k}^+} \Bigr)^r
w(F_{i,k}) \\
&= 8^r \sum_{i,k} \bigl( T(f\sigma^{-1})(i,k) \bigr)^r \mu(i,k) ,
\end{align*}
where
\[
Tg(i,k) = \frac{1}{\sigma(R_{i,k}^+)} \int_{S_{i,k}^+} \lvert g \rvert \, \sigma
\quad\text{and}\quad
\mu(i,k) = \Bigl( \lvert R_{i,k}^+ \rvert^\alpha \sigma_{R_{i,k}^+} \Bigr)^r w(F_{i,k}),
\]
and $R_{i.k}^+$ is defined as in~\eqref{eq:subset}.  

The operator $T$ is a linear operator defined on functions over the measure space $(\mathbb{Z}\times\mathbb{Z}, \mu)$.
We immediately have that $T$ is bounded from $L^\infty(\sigma)$ to $l^\infty(\mu)$, since
\begin{align*}
\norm{Tg}_{l^\infty(\mu)} = 
\sup_{i,k} \frac{1}{\sigma(R_{i,k}^+)} \int_{S_{i,k}^+} \lvert g \rvert \, \sigma \leq \norm{g}_{L^\infty(\sigma)} .
\end{align*}
We now claim  that $T$ is bounded from $L^{1}(\sigma)$ to $l^{1,\infty}(\mu)$.
Let $\{S^+_j\}_j$ denote the maximal disjoint collection of dyadic parabolic rectangles in
\[
\Bigl\{ S_{i,k}^+: Tg(i,k) = \frac{1}{\sigma(R_{i,k}^+)} \int_{S_{i,k}^+} \lvert g \rvert \, \sigma > \lambda \Bigr\} .
\]
For every $j$, define
\[
I_j = \bigl\{ (i,k): Tg(i,k)>\lambda , S_{i,k}^+ \subset S_{j}^+ \bigr\} .
\]
We then have that
\begin{equation}
\label{eq:dyarect_div}
\{(i,k): Tg(i,k)>\lambda\} = \bigcup_j I_j.
\end{equation}
Note that $S_{i,k}^+ \subset S_{j}^+$ implies $R_{i,k}^+ \subset R_{j}^+$ by \eqref{eq:subset}.
For every $(i,k)\in I_j$ and $z\in F_{i,k}\subset S_{i,k}^- \subset R_{i,k}^-$, we have
\[
\lvert R_{i,k}^+ \rvert^\alpha
\dashint_{R_{i,k}^+} \sigma 
\leq 
M_{\alpha,d}^{+}(\sigma \chi_{R_{i,k}^+})(z)
\leq
M_{\alpha,d}^{+}(\sigma \chi_{R_{j}^+})(z) .
\]
Thus, we obtain
\begin{align*}
\sum_{(i,k) \in I_j} \mu(i,k) =
\sum_{(i,k) \in I_j} \int_{F_{i,k}} \lvert R_{i,k}^+ \rvert^{\alpha r} (\sigma_{R_{i,k}^+})^r \, w
\leq
\sum_{(i,k) \in I_j} \int_{F_{i,k}} (M_{\alpha,d}^{+}(\sigma \chi_{R_{j}^+}))^r \, w .
\end{align*}
Observe that for $(i,k)\in I_j$ we have $F_{i,k} \subset S_{i,k}^- \subset S_j \subset R_j$.
Therefore, by \eqref{eq:dyarect_div}, using that the $F_{i,k}$ are pairwise disjoint, and applying our hypothesis, we get
\begin{align*}
\sum_{\{(i,k): Tg(i,k)>\lambda\}} \mu(i,k) 
&= \sum_j \sum_{(i,k) \in I_j} \mu(i,k) \\
& \leq 
\sum_j \sum_{(i,k) \in I_j} \int_{F_{i,k}} (M_{\alpha,d}^{+}(\sigma \chi_{R_{j}^+}))^r \, w \\
&\leq
\sum_j \int_{R_j} (M_{\alpha,d}^{+}(\sigma \chi_{R_{j}^+}))^r \, w \\
& \leq 
[(w,v)]_{S_{q,r,\alpha}^+}^r \sum_j \sigma(R_{j}^+)^\frac{r}{q}
\\
&\leq 
[(w,v)]_{S_{q,r,\alpha}^+}^r \biggl( \sum_j \sigma(R_{j}^+)  \biggr)^\frac{r}{q} \\
& \leq
[(w,v)]_{S_{q,r,\alpha}^+}^r
\biggl( \sum_j \frac{1}{\lambda} \int_{S_{j}^+} \lvert g \rvert \, \sigma  \biggr)^\frac{r}{q}
\\
&\leq
[(w,v)]_{S_{q,r,\alpha}^+}^r
\biggl( \frac{1}{\lambda} \int_{\mathbb{R}^{n+1}} \lvert g \rvert \, \sigma  \biggr)^\frac{r}{q} .
\end{align*}
Hence, $T$ is bounded from $L^{1}(\sigma)$ to $l^{r/q,\infty}(\mu)$.
Therefore, by the  Marcinkiewicz interpolation theorem, we have  that $T$ is bounded from $L^q(\sigma)$ to $l^r(\mu)$, that is,
there exists a constant $C_2=C_2(q,r)$ such that
\begin{multline*}
\biggl(
\sum_{i,k} \bigl( T(f\sigma^{-1})(i,k) \bigr)^r \mu(i,k)
\biggr)^\frac{1}{r} \\
\leq
C_2 [(w,v)]_{S_{q,r,\alpha}^+}
\biggl( \int_{\mathbb{R}^{n+1}} \lvert f \sigma^{-1} \rvert^q \sigma \biggr)^\frac{1}{q}
=
C_2 [(w,v)]_{S_{q,r,\alpha}^+} \biggl( \int_{\mathbb{R}^{n+1}} \lvert f \rvert^q \, v \biggr)^\frac{1}{q};
\end{multline*}
thus,  we have shown that
\begin{equation*}
\biggl( \int_{\mathbb{R}^{n+1}} (M_{\alpha,d}^{+}f)^r \, w \biggr)^\frac{1}{r}
\leq
8 \biggl( \sum_{i,k} \bigl( T(f\sigma^{-1})(i,k) \bigr)^r \mu(i,k) \biggr)^\frac{1}{r}
= 8 C_2 [(w,v)]_{S_{q,r,\alpha}^+} \biggl( \int_{\mathbb{R}^{n+1}} \lvert f \rvert^q \, v \biggr)^\frac{1}{q} .
\end{equation*}
This completes the proof.
\end{proof}

\section*{Declarations}
\noindent\textbf{Conflict of interest} The authors have no relevant financial or non-financial interests to disclose.

\end{document}